\documentclass[11pt,a4paper]{article}
\usepackage{latexsym,amssymb,amsmath,amsthm}
\usepackage{ascmac,calc,mathrsfs, url, braket, comment}
\usepackage{setspace}
\usepackage{graphicx}

\newtheorem{theorem}{Theorem}[section]
\newtheorem{lemma}[theorem]{Lemma}

\newtheorem{corollary}[theorem]{Corollary}
\newtheorem{proposition}[theorem]{Proposition}

\newtheorem{claim}[theorem]{Claim}

\theoremstyle{definition}

\newtheorem{remark}[theorem]{Remark}
\newtheorem{example}[theorem]{Example}

\newcommand{\Z}{\mathbf{Z}}
\newcommand{\N}{\mathbf{N}}

\newcommand{\F}{\mathcal{F}}

\newcommand{\E}{\mathcal{E}}
\newcommand{\BP}{\mathcal{P}}

\newcommand{\finbox}{}
\renewcommand{\phi}{\varphi}

\makeatletter
\@addtoreset{equation}{section}

\makeatother

\begin{document}
\title{List Supermodular Coloring with Shorter Lists}
\author{
Yu Yokoi
\thanks{National Institute of Informatics, Tokyo 101-8430, Japan. 
E-mail: {\tt yokoi@nii.ac.jp}.
}
}
\date{\today}
\maketitle
\setstretch{1.05}
\vspace{-5mm}
\begin{abstract}
In 1995, Galvin proved that a bipartite graph $G$ admits a list edge coloring 
if every edge is assigned a color list of length $\Delta(G)$, the maximum degree of the graph.
This result was improved by Borodin, Kostochka and Woodall, who proved that
$G$ still admits a list edge coloring if 
every edge $e=st$ is assigned a list of $\max\{d_{G}(s), d_{G}(t)\}$ colors.
Recently, Iwata and Yokoi provided the list supermodular coloring theorem, 
that extends Galvin's result to the setting of Schrijver's supermodular coloring.
This paper provides a common generalization of these two extensions of Galvin's result.
\end{abstract}

\section{Introduction}
An {\em edge coloring} of an undirected graph%
\footnote{In this paper, a bipartite graph means a bipartite multigraph, i.e., parallel edges are allowed.}
 is a function that assigns a color to each edge 
so that no two adjacent edges have the same color.
In 1995, Galvin \cite{Galvin95} proved the list edge coloring conjecture for bipartite graphs,
which states that the list edge chromatic number of a bipartite graph equals its edge chromatic number.   
Since K\H{o}nig's theorem  \cite{Konig16} states that the edge chromatic number of a bipartite graph $G$ equals 
its maximum degree $\Delta (G)$, Galvin's result is written as follows.  
\begin{theorem}[Galvin \cite{Galvin95}]\label{thm:Galvin}
	For a bipartite graph $G$,
	if each edge $e$ has a list $L(e)$ of $\Delta(G)$ colors, 
	then there exists an edge coloring such that every edge $e$ is assigned a color in $L(e)$.
	\finbox
\end{theorem}
Exploiting Galvin's proof technique for Theorem \ref{thm:Galvin}, 
Borodin, Kostochka and Woodall \cite{BKW97} showed the following stronger version,
in which some elements may have shorter lists.
For a bipartite graph $G$ and a vertex $s$, we denote by $d_{G}(s)$ the degree of $s$ in $G$. 
\begin{theorem}[Borodin, Kostochka and Woodall \cite{BKW97}]\label{thm:BKW}
	For a bipartite graph $G=(S, T; E)$,
	if each edge $e=st$ has a list $L(e)$ of $\max\{d_{G}(s), d_{G}(t)\}$ colors, 
	then there exists an edge coloring such that every edge $e$ is assigned a color in $L(e)$.
	\finbox
\end{theorem}

Another generalization of Galvin's result is the list supermodular coloring theorem of
Iwata and Yokoi \cite{IY17}.
It extends Theorem~\ref{thm:Galvin} to the setting of Schrijver's supermodular coloring \cite{Schr85},
which is described below.
(We mostly use the same notations as in \cite{IY17}.)

Let $U$ be a finite set.  
We say that $X, Y\subseteq U$ are {\em intersecting}
if none of $X\cap Y$, $X\setminus Y$ and $Y\setminus X$ are empty. 
A family $\F\subseteq 2^{U}$ is called an {\em intersecting family} if
every intersecting pair of $X, Y\in \F$ satisfies $X\cup Y, X\cap Y\in \F$.
A set function $g$ on $\F$ is called  
{\em intersecting-supermodular} if $\F$ is 
an intersecting family and $g$ satisfies the {\em supermodular inequality}
$g(X)+g(Y)\leq g(X\cup Y)+g(X\cap Y)$
for every intersecting pair of $X, Y\in \F$.
For any positive integer $k\in \N$, we write $[k]:=\{1,2,\dots,k\}$.
A function $\pi:U\to[k]$ {\em dominates} $g:\F\to \Z$
if $|\pi(X)|\geq g(X)$ holds for every $X\in \F$, where $\pi(X):=\set{\pi(u)|u\in X}$.

Let $g_{1}:\F_{1}\to \Z$ and $g_{2}:\F_{2}\to \Z$  be intersecting-supermodular functions.
A function $\pi:U\to [k]$ is called a {\em supermodular $k$-coloring} for $(g_{1},g_{2})$
if it dominates both $g_{1}$ and $g_{2}$.
Let us assume $|X|\geq g_{i}(X)$ for every $i\in\{1,2\}$ and $X\in \F_{i}$.
This condition is clearly necessary for the existence of a supermodular $k$-coloring. 
Schrijver \cite{Schr85} showed that, under this assumption, the minimum $k\in \N$ that admits a supermodular $k$-coloring of $(g_{1}, g_{2})$
equals $\Delta(g_{1}, g_{2})$, where
\[\Delta(g_{1}, g_{2}):=\max\{1, \max\set{g_{i}(X)|i\in \{1,2\},~X\in \F_{i}}\}.\]

As a list coloring version of the Schrijver's result, Iwata and Yokoi \cite{IY17} proved the following list supermodular coloring theorem.
Let $\Sigma$ be a set of colors and $L:U\to 2^{\Sigma}$ be an assignment of color lists to elements. 
A function $\phi:U\to \Sigma$ is called
a {\em list supermodular coloring}  for $(g_{1}, g_{2}, L)$ if  
$\phi$ dominates both $g_{1}$ and $g_{2}$ and every $u\in U$ satisfies $\phi(u)\in L(u)$.

\begin{theorem}[Iwata and Yokoi \cite{IY17}]\label{thm:IY}
Let $g_{1}:\F_{1}\to \Z$ and $g_{2}:\F_{2}\to \Z$ be
intersecting-supermodular functions 
such that $|X|\geq g_{i}(X)$ for every $i\in\{1,2\}$ and $X\in \F_{i}$.
If each element $u\in U$ has a list $L(u)$ of $\Delta(g_{1}, g_{2})$ colors, then
there exists a list supermodular coloring for $(g_{1}, g_{2}, L)$. 
\finbox
\end{theorem}
The aim of this paper is to provide a common generalization of Theorems~\ref{thm:BKW} and \ref{thm:IY}.

For an intersecting supermodular function $g:\F\to \Z$ on $U$, 
define the family $\E[g]\subseteq \F$ by
\[\E[g]:=\set{X\in \F|g(X)\geq 2,~\not\exists X'\in \F: [X'\subsetneq X, ~g(X')\geq g(X)]}.\] 
We call $\E[g]$ the {\em effective set family} of $g$. 
A function $d[g]:U\to \N$ is defined by  
\[d[g](u)=\max\{1, \max\set{g(X)|u\in X\in \E[g]}\},\]
where $\max\set{g(X)|u\in X\in \E[g]}$ takes $-\infty$ if 
$\set{X|u\in X\in \E[g]}=\emptyset$.

\begin{theorem}[\bf Main Theorem]\label{thm:main}
Let $g_{1}:\F_{1}\to \Z$ and $g_{2}:\F_{2}\to \Z$ be
intersecting-supermodular functions 
such that $|X|\geq g_{i}(X)$ for every $i\in\{1,2\}$ and $X\in \F_{i}$.
If each element $u\in U$ has a list $L(u)$ of $\max\{d[g_{1}](u), d[g_{2}](u)\}$ colors, then 
there exists a list supermodular coloring for $(g_{1}, g_{2}, L)$. 
\finbox
\end{theorem}
By definition, $\max\{d[g_{1}](u), d[g_{2}](u)\}\leq \Delta(g_{1}, g_{2})$ for every $u\in U$,
and hence Theorem~\ref{thm:main} is an extension of Theorem~\ref{thm:IY}.
Also, Theorem~\ref{thm:main} is a generalization of Theorem~\ref{thm:BKW} as follows: 
For a bipartite graph $G=(S, T;E)$, 
let $\F_{1}:=\set{\delta_{G}(s)|s\in S}$ and $\F_{2}:=\set{\delta_{G}(t)|t\in T}$,
where $\delta_{G}(s)\subseteq E$ is the set of edges incident to $s$.
Define $g_{1}:\F_{1}\to \Z$ by $g_{1}(\delta_{G}(s))=|\delta_{G}(s)|$ and 
$g_{2}:\F_{2}\to \Z$ by $g_{2}(\delta_{G}(t))=|\delta_{G}(t)|$.
Then, $g_{1}$, $g_{2}$ are intersecting-supermodular functions on $E$.
We see that an edge coloring of $G$ with $k$ colors 
is just a supermodular $k$-coloring for $(g_{1}, g_{2})$.
Moreover, for each edge $e=st$, 
the value $\max\{d[g_{1}](e), d[g_{2}](e)\}$ coincides with $\max\{d_{G}(s), d_{G}(t)\}$.
Thus, Theorem~\ref{thm:main} generalizes both Theorems~\ref{thm:BKW} and \ref{thm:IY}.

Observe that, for every $u\in U$, the definition of $d[g_{i}]$ implies
\[\max\{d[g_{1}](u), d[g_{2}](u)\}\leq \max\{1, \max\set{g_{i}(X)|i\in \{1,2\},~u\in X\in \F_{i}}\}.\]
Theorem~\ref{thm:main} then immediately implies the following corollary,
which is weaker than Theorem~\ref{thm:main} but can be stated without introducing  $\E[g_{i}]$ nor $d[g_{i}]$.

\begin{corollary}
\label{cor:main}
Let $g_{1}:\F_{1}\to \Z$ and $g_{2}:\F_{2}\to \Z$ be
intersecting-supermodular functions 
such that $|X|\geq g_{i}(X)$ for every $i\in\{1,2\}$ and $X\in \F_{i}$.
If each element $u\in U$ has a list $L(u)$ of $\max\{1, \max\set{g_{i}(X)|i\in \{1,2\},~u\in X\in \F_{i}}\}$ colors, then 
there exists a list supermodular coloring for $(g_{1}, g_{2}, L)$. 
\finbox
\end{corollary}
Corollary~\ref{cor:main} is properly weaker than Theorem~\ref{thm:main}.
(For example, let $g_{1},g_{2}:\F\to \Z$ be the same function 
such that $\F=\{X', X\}$, $X'\subsetneq X$, $|X'|\geq 2$ 
and $g_{i}(X')=g_{i}(X)=2$. 
Then for an element $u\in X\setminus X'$, 
we have $\max\{1, \max\set{g_{i}(X)|i\in \{1,2\},~u\in X\in \F_{i}}\}=2$ 
while $\max\{d[g_{1}](u), d[g_{2}](u)\}=1$.)
However, we can see that Corollary~\ref{cor:main} is still a common generalization of
Theorems~\ref{thm:BKW} and \ref{thm:IY}.

\smallskip
We prove Theorem~\ref{thm:main} by combining ideas of
Iwata and Yokoi \cite{IY17} and of Borodin et al. \cite{BKW97}.
By the result of Iwata and Yokoi (to be described in Proposition~\ref{prop:IY}), we can reduce the problem of finding
a list supermodular coloring to a problem of finding a pair of auxiliary functions
satisfying certain conditions.
We then show the existence of such auxiliary functions (Lemma~\ref{lem:key}), 
which is the main part of this paper.
The proof of this lemma is by induction on the ground set. 
For that, we construct a special bipartite graph induced from
the pair of intersecting-supermodular functions and apply a method (Proposition~\ref{prop:BKW}) used by Borodin et al. for bipartite edge coloring.
\smallskip

The rest of this paper is organized as follows.
In Section~\ref{sec:Key},
we introduce a key lemma, from which Theorem~\ref{thm:main} follows.
To prove this lemma, Section~\ref{sec:BPintro} introduces the notion of 
``bunch partitions'' defined for intersecting-supermodular functions.
There, we provide their properties, but proofs are postponed to Section~\ref{sec:BP}.
Using bunch partitions, Section~\ref{sec:KeyLem} shows the key lemma.

\section{Key Lemma}\label{sec:Key}
In the proof of Theorem~\ref{thm:IY}, 
Iwata and Yokoi proved the following proposition,
which describes a sufficient condition for the existence of a list supermodular coloring 
in terms of two auxiliary functions $\pi_{1}$ and $\pi_{2}$.

\begin{proposition}[Iwata and Yokoi {\cite[Proposition 3.2]{IY17}}]\label{prop:IY}
Let $g_{1}:\F_{1}\to \Z$ and $g_{2}:\F_{2}\to \Z$ be
intersecting-supermodular functions. 
For an arbitrary function $f:U\to \N$, 
assume that there exist functions $\pi_{1}, \pi_{2}:U\to \N$ satisfying the following conditions.
\begin{enumerate}
	\setlength{\itemsep}{0mm}
	\item[\rm (i)] For every $u\in U$, we have~ $\pi_{1}(u)+\pi_{2}(u)-1\leq f(u)$.
	\item[\rm (ii)] For each $i\in\{1,2\}$,  $\pi_{i}$ dominates $g_{i}$. 
	\end{enumerate}
Then, for any  $L:U\to 2^{\Sigma}$ such that $|L(u)|=f(u)~(\forall u\in U)$, 
there exists a list supermodular coloring for $(g_{1},g_{2}, L)$.
\finbox
\end{proposition}
Suppose that $(g_{1}, g_{2}, L)$ is provided.
Proposition~\ref{prop:IY} says that,  to show the existence of a list supermodular coloring, 
it suffices to find auxiliary functions $\pi_{1}, \pi_{2}$ satisfying conditions (i) and (ii)
for $f$ such that $f(u)=|L(u)|$ for each $u\in U$.
Indeed, Iwata and Yokoi proved Theorem~\ref{thm:IY} by showing the existence of 
such $\pi_{1}, \pi_{2}$ for the constant function $f=\Delta(g_{1}, g_{2})$.
In this case, the construction of $\pi_{1}$ and $\pi_{2}$ can be easily done by using Schrijver's result
(see Remark~\ref{rem:IY}).

In this paper, we deduce Theorem~\ref{thm:main} from Proposition~\ref{prop:IY}
by constructing $\pi_{1}$ and $\pi_{2}$ in a more careful manner.
We show the following lemma.
\begin{lemma}[\bf Key Lemma]\label{lem:key}
For any
intersecting-supermodular functions 
$g_{1}:\F_{1}\to \Z$ and $g_{2}:\F_{2}\to \Z$ 
such that $|X|\geq g_{i}(X)$ for every $i\in\{1,2\}$ and $X\in \F_{i}$,
there exist functions $\pi_{1}, \pi_{2}:U\to \N$ satisfying {\rm (i)} and {\rm (ii)}
with $f(u)$ defined by $f(u)=\max\{d[g_{1}](u), ~d[g_{2}](u)\}$.
\finbox
\end{lemma}
Once Lemma~\ref{lem:key} is proved, we can immediately obtain Theorem~\ref{thm:main} 
by combining it with Proposition~\ref{prop:IY}.
The remainder of this paper is devoted to showing Lemma~\ref{lem:key}.

\begin{remark}\label{rem:IY}
For reference, we provide the method of Iwata and Yokoi \cite{IY17}
for a construction of $\pi_{1}$ and $\pi_{2}$ satisfying (i) and (ii) with $f$ defined by $f(u)=\Delta(g_{1}, g_{2})$ for every $u\in U$.
Let $k:=\Delta(g_{1}, g_{2})$.
Take a supermodular $k$-coloring $\pi:U\to[k]$ for $(g_{1}, g_{2})$,
whose existence is guaranteed by the result of Schrijver \cite{Schr85}.  
Define $\pi_{1}, \pi_{2}:U\to [k]$ by $\pi_{1}(u):=\pi(u)$ and $\pi_{2}(u):=k+1-\pi(u)$ for each $u\in U$.
Then (i) holds because $\pi_{1}(u)+\pi_{2}(u)-1=k=f(u)$ for every $u$.
Also (ii) holds as $g_{i}(X)\leq |\pi(X)|=|\pi_{1}(X)|=|\pi_{2}(X)|$ for any $i\in\{1,2\}$ and $X\in \F_{i}$.
\end{remark}
\section{Bunch Partitions}\label{sec:BPintro}
To show Lemma~\ref{lem:key},
this section introduces the notion of  bunch partitions for intersecting-supermodular functions.
This structure connects our supermodular coloring setting 
to a technique on bipartite graphs used by Borodin et al. \cite{BKW97}.

Let $g:\F\to \Z$ be an intersecting-supermodular function on $U$. 
For a subset $K\subseteq U$, the {\em reduction} of $g$ by $K$ is 
the function $g_{K}:\F_{K}\to \Z$ defined by
$\F_{K}=\set{Z\setminus K|Z\in \F}$ and
\begin{align*}
g_{K}(X)=\max\set{\hat{g}_{K}(Z)|Z\in \F,~ Z\setminus K=X}\quad (X\in \F_{K}),
\end{align*}
where
$\hat{g}_{K}(Z)=g(Z)-1$ for $Z\in \F$ with $Z\cap K\neq\emptyset$
and $\hat{g}_{K}(Z)=g(Z)$ for $Z\in \F$ with $Z\cap K=\emptyset$.
The following fact is known (cf., Frank \cite{Frankbook}, Iwata and Yokoi \cite[Claim 2.1]{IY17}).
\begin{claim}\label{claim:reduction}
	The reduction $g_{K}:\F_{K}\to \Z$ is an intersecting-supermodular function.
\finbox
\end{claim}

Recall that the effective set family $\E[g]$ is a collection of subsets $X\in \F$
such that $g(X)\geq 2$ and no proper subset $X'\in \F$ satisfies $g(X')\geq g(X)$.
From $\E[g]$,
define a family $\BP[g]$ by
\[\textstyle{\BP[g]:=\set{X\in \E[g]|\not\exists X'\in \E[g]: X\subsetneq X'}\cup \set{\{u\}|u\in U\setminus \bigcup \E[g]}}.\]
That is, $\BP[g]$ contains
all maximal members of $\E[g]$ and singleton sets of elements not included in any member of $\E[g]$ (see Figure~\ref{fig1} for an example).
From the definitions of $\E[g]$ and $\BP[g]$, we can obtain the following two claims.
The first is clear by definition.
\begin{claim}\label{claim:BP01}
For any $X\in \E[g]$, there exists $P\in \BP[g]\cap \E[g]$ such that $X\subseteq P$.
\finbox
\end{claim}
\begin{claim}\label{claim:BP02}
For any $X\in \F$ such that $X\not\in \E[g]$ and $g(X)\geq 2$,
there exist $X'\in \E[g]$ and $P\in \BP[g]\cap \E[g]$ such that
$X'\subseteq X\cap P$ and $g(X')\geq g(X)$.
\end{claim}
\begin{proof}
For such $X$, let $X'$ be a minimal maximizer of $g$ subject to  $X'\subseteq X$.
Then, $X\not\in \E[g]$ implies  $X'\subsetneq X$ and $g(X')\geq g(X)\geq 2$.
By definition, $X'$ belongs to $\E[g]$.
Then, by Claim~\ref{claim:BP01}, there is $P\in \BP[g]\cap \E[g]$ with $X'\subseteq P$. 
They satisfy $X'\subseteq X\cap P$ and $g(X')\geq g(X)$.
\end{proof}

We provide the following four properties of $\BP[g]$,
whose proofs are postponed to Section~\ref{sec:BP}.

\begin{proposition}\label{prop:BP}
The family $\BP[g]$ is a partition of $U$.
\finbox
\end{proposition}
We call $\BP[g]$ the {\em bunch partition} of $U$ by $g$.
We denote by $P[g](u)$ the unique part containing $u\in U$.
Recall that $d[g]$ is defined by $d[g](u)=\max\{1, \max\set{g(X)|u\in X\in \E[g]}\}$.

\begin{proposition}\label{prop:BP2}
For every $u\in \bigcup \E[g]$, we have $P[g](u)\in \E[g]$ and $d[g](u)=g(P[g](u))\geq 2$.
For every $u\in U\setminus \bigcup \E[g]$, we have $P[g](u)=\{u\}$ and $d[g](u)=1$.
\finbox
\end{proposition}

A subset $K\subseteq U$ is called a {\em partial transversal} 
of $\BP[g]$ if
$|K\cap P|\leq 1$ 
for every $P\in \BP[g]$.
\begin{proposition}\label{prop:BP3}
Suppose $|Z|\geq g(Z)$ for every $Z\in \F$. 
Then, for any partial transversal $K$ of $\BP[g]$,
the reduction $g_{K}$ satisfies $|X|\geq g_{K}(X)$ for every $X\in \F_{K}$.
\finbox
\end{proposition}
\begin{proposition}\label{prop:BP4}
Suppose $|Z|\geq g(Z)$ for every $Z\in \F$ and 
take an arbitrary partial transversal $K$ of $\BP[g]$.
For every $u\in U\setminus K$, we see that
$P[g](u)\cap K=\emptyset$ implies $d[g_{K}](u)=d[g](u)$
and $P[g](u)\cap K\neq\emptyset$ implies  $d[g_{K}](u)<d[g](u)$.
\finbox
\end{proposition}
To capture the notion of bunch partitions, we provide an example.
\begin{figure}[htbp]
\begin{center}
   \includegraphics[width=140mm]{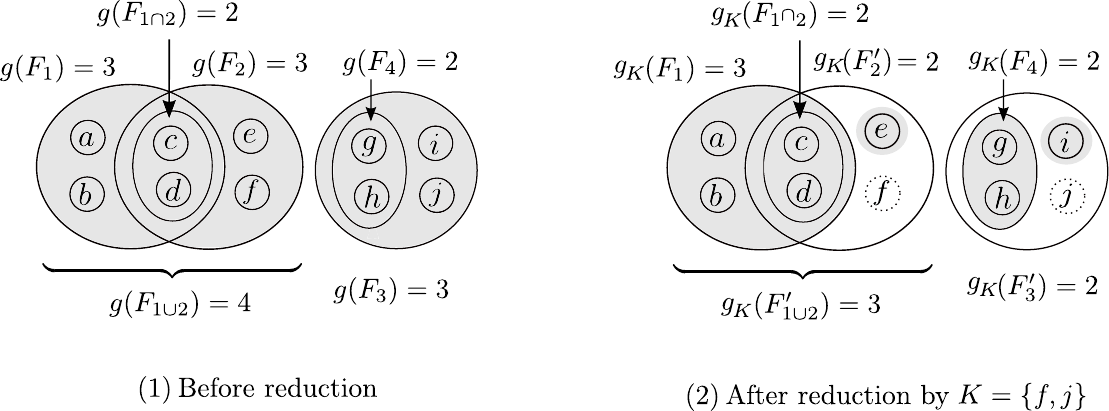}
\caption{
\small\baselineskip=10pt 
(1) The family $\F$ in Example~\ref{ex:example1} and values of $g$. 
(2) The family $\F_{K}$ and values of $g_{K}$, where $F'$ denotes $F\setminus K$ for each $F\in \F$.
The families $\BP[k]$ and $\BP[g_{K}]$ are represented by gray color.}
\label{fig1}
\end{center}
\end{figure}

\begin{example}\label{ex:example1}
Let $U=\{a,b,c,d,e,f,g,h,i,j\}$
and $\F=\{F_{1}, F_{2}, F_{1\cup 2}, F_{1\cap 2}, F_{3}, F_{4}\}$, where
$F_{1}=\{a,b,c,d\}$,
$F_{2}=\{c,d,e,f\}$,
$F_{1\cup 2}=F_{1}\cup F_{2}$, 
$F_{1\cap 2}=F_{1}\cap F_{2}$,
$F_{3}=\{g,h,i,j\}$,
$F_{4}=\{g,h\}$.
Define $g:\F\to \Z$ by 
$g(F_{1})=g(F_{2})=g(F_{3})=3$, 
$g(F_{1\cup 2})=4$, 
$g(F_{1\cap 2})=g(F_{4})=2$ (see Figure~\ref{fig1}).
Then,  $\E[g]=\F$, $\BP[g]=\{F_{1\cup 2}, F_{3}\}$.
Also, $d[g](u)=4$ for each $u\in F_{1\cup 2}$ and  
$d[g](u)=3$ for each $u\in F_{3}$.

Let $K=\{f,j\}$ and 
$g_{K}$ be the reduction of $g$ by $K$.
Then, $\E[g_{K}]=\{F_{1}, F_{1\cap 2},F_{4}\}$ and $\BP[g]=\{F_{1}, \{e\},F_{4},\{i\}\}$,
Also, $d[g](u)=3$ for each $u\in F_{1}$,  
$d[g](u)=2$ for each $u\in F_{4}$, and $d[g](e)=d[g](i)=1$.
\end{example}

\begin{remark}
The definition of a bunch partition is 
similar to that of a {\em solid partition} studied by B\'{a}r\'{a}sz, Becker, and Frank \cite{BBF05}. 
For a directed graph,  a vertex set $X$ is called {\em in-solid}  
if there is no nonempty proper subset $X'$ of $X$ satisfying $\varrho(X')\leq \varrho(X)$,
where $\varrho(X)$ is the in-degree, i.e., the number of edges entering $X$.
Using the submodularity of $\varrho$, it was shown in \cite{BBF05} that
maximal in-solid sets form a partition.
(The same is true for {\em out-solid} sets, which are defined analogously by the out-degrees.)
Proposition~\ref{prop:BP} will be shown by similar arguments (in Section~\ref{sec:BP}). 
\end{remark}

\section{Proof of the Key Lemma}\label{sec:KeyLem}
This section proves Lemma~\ref{lem:key} relying on Propositions~\ref{prop:BP}--\ref{prop:BP4}.
(In fact, we prove a stronger version of Lemma~\ref{lem:key}.)
First, we introduce the following fact on bipartite graphs,
which was used by Borodin et al. \cite{BKW97} to show Theorem~\ref{thm:BKW}.
For completeness, we provide their proof.
\begin{proposition}[Borodin et al. {\cite[Lemma 3.1]{BKW97}}]\label{prop:BKW}
	For a bipartite graph $G=(S,T;E)$, if $|S|\geq |T|$ and $S$ contains no isolated vertices, 
	then $G$ has a nonempty matching $M\subseteq E$ such that
	every $e=st\in E$ with  $s\in \partial M$ satisfies $t\in \partial M$,
	where $\partial M\subseteq S\cup T$ is the set of vertices incident to some edge in $M$.
\end{proposition}
\begin{proof}[Proof \cite{BKW97}]
Since $|T|\leq |S|$, 
we have $|\Gamma(S)|\leq |T|\leq |S|$, where $\Gamma(S)$ denotes the set of neighbors of vertices in $S$.
Let $V$ be a minimal nonempty subset of $S$ such that $|\Gamma(V)|\leq |V|$. 
Then $|\Gamma(V)|=|V|$ and there is a matching $M$ with $\partial M\cap S=V$. 
(If $|V|=1$, this holds because $S$ contains no isolated vertices. 
If $|V|\geq 2$, the minimality of $V$ implies $|\Gamma(V)|=|V|$ and $|\Gamma(W)|>|W|$ whenever $\emptyset \subsetneq W\subsetneq V$,
and hence the K\H{o}nig-Hall theorem implies that there is a matching $M$ with $\partial M\cap S=V$.) 
By $\partial M\cap T\subseteq \Gamma(V)$ and $|\partial M\cap T|=|\partial M\cap S|=|V|=|\Gamma(V)|$, we obtain 
$\partial M\cap T=\Gamma(V)$. We see that this $M$ has the required property. 
\end{proof}

For the pair of partitions of the same ground set, 
Proposition~\ref{prop:BKW} is rephrased as follows.
\begin{claim}\label{claim:matching}
Let $g_{1}:\F_{1}\to \Z$ and $g_{2}:\F_{2}\to \Z$ be intersecting-supermodular functions.
For the bunch partitions $\BP[g_{1}]$ and $\BP[g_{2}]$, there exists a nonempty common partial transversal $K\subseteq U$
satisfying one or both of the following. 
\begin{itemize}
	\setlength{\itemsep}{0mm}
\item[\rm (a)] Every $u\in U$ with $P[g_{1}](u)\cap K\neq \emptyset$ satisfies $P[g_{2}](u)\cap K\neq \emptyset$.
\item[\rm (b)] Every $u\in U$ with $P[g_{2}](u)\cap K\neq \emptyset$ satisfies $P[g_{1}](u)\cap K\neq \emptyset$.
\end{itemize}
\end{claim}
\begin{proof}
Let 
us denote the bunch partitions $\BP[g_{1}]$ and $\BP[g_{2}]$ by
\[\BP[g_{1}]=\{P_{1}^{1}, P_{2}^{1},\dots, P_{q_{1}}^{1}\},~~~~ \BP[g_{2}]=\{P_{1}^{2}, P_{2}^{2},\dots, P_{q_{2}}^{2}\}.\]
Let $G'=(S',T'; E')$ be a bipartite graph such that each vertex corresponds to 
a part of partitions and each edge corresponds to an element.
That is, we define
\begin{align*}
& S'=\set{s_{j}|j=1,2,\dots,q_{1}},\quad T'=\set{t_{k}|k=1,2,\dots, q_{2}},\\
& E'=\set{e_{u}=s_{j}t_{k}| u\in U, ~P[g_{1}](u)=P_{j}^{1}, ~P[g_{2}](u)=P_{k}^{2}}.
\end{align*}
There is a one-to-one correspondence between $U$ and $E'$.
Clearly, $G'$ has no isolated vertex.

We first consider the case $|S'|\geq |T'|$.
In this case,  apply Proposition~\ref{prop:BKW} with $S=S'$ and $T=T'$.
Then, there is a nonempty matching $M\subseteq E'$
such that every $e_{u}=s_{j}t_{k}\in E'$ with $s_{j}\in \partial M$ satisfies $t_{k}\in \partial M$.
Let $K:=\set{u\in U|e_{u}\in M}$.
As $M$ is a nonempty matching, $K$ is a nonempty common partial transversal of $\BP[g_{1}]$ and $\BP[g_{2}]$.
Also, the condition of $M$ means that every $u\in U$ with
$P[g_{1}](u)\cap K\neq \emptyset$ satisfies $P[g_{2}](u)\cap K\neq \emptyset$.
Thus, (a) holds.

In the case $|S'|<|T'|$, 
by applying Proposition~\ref{prop:BKW} with $S=T'$ and $T=S'$,
we can similarly obtain 
a nonempty common partial transversal $K$ satisfying (b).
\end{proof}

We are now ready to prove Lemma~\ref{lem:key}.
Actually, we show the following stronger statement, 
in which an additional constraint (iii) is also required for the functions $\pi_{1}$ and $\pi_{2}$. 
It is clear that Lemma~\ref{lem:key2} implies Lemma~\ref{lem:key}. 
\begin{lemma}[\bf Stronger Version of Lemma \ref{lem:key}]\label{lem:key2}
	For any
	intersecting-supermodular functions 
	$g_{1}:\F_{1}\to \Z$ and $g_{2}:\F_{2}\to \Z$ 
	such that $|X|\geq g_{i}(X)$ for every $i\in\{1,2\}$ and $X\in \F_{i}$,
	there exist functions $\pi_{1}, \pi_{2}:U\to \N$ satisfying the following {\rm (i)}, {\rm (ii)}, and {\rm (iii)}. 
	\begin{enumerate}
		\setlength{\itemsep}{0mm}
		\item[\rm (i)] For every $u\in U$, we have~ $\pi_{1}(u)+\pi_{2}(u)-1\leq \max\{d[g_{1}](u), ~d[g_{2}](u)\}$. 
		\item[\rm (ii)] For each $i\in\{1,2\}$,  $\pi_{i}$ dominates $g_{i}$.  
		\item[\rm (iii)] For each $i\in \{1,2\}$ and $u\in U$, we have	$\pi_{i}(u)\leq d[g_{i}](u)$.
	\end{enumerate}
\end{lemma} 
%

\begin{proof}[Proof of Lemma~\ref{lem:key2} (and hence of Lemma~\ref{lem:key})]
We use induction on $|U|$, i.e., the size of the ground set. 
Recall that  $d[g_{i}](u) =\max\{1, \max\set{g_{i}(Z)|u\in Z\in \E[g_{i}]}\}$.
\medskip

First, consider the case $|U|=1$, i.e., $U=\{u\}$. 
Since $g_{i}(Z)\leq |Z|\leq 1$
for every $i\in \{1,2\}$ and $X\in \F_{i}$,
we have $d[g_{1}](u)=d[g_{2}](u)=1$. 
Let $\pi_{1}(u)=\pi_{2}(u)=1$. 
Then, we can observe that (i), (ii), and (iii) are all satisfied.
\medskip

We now consider the case $|U|>1$. 
By Claim~\ref{claim:matching}, there is a nonempty common partial transversal $K$ 
of $\BP[g_{1}]$ and $\BP[g_{2}]$ satisfying (a) or (b).
For each $i\in \{1,2\}$, we denote by $g_{i}'$ the reduction of $g_{i}$ by $K$.
The domain of $g'_{i}$ is denoted by $\F'_{i}=\set{Z\setminus K|Z\in \F_{i}}$.
By Claim~\ref{claim:reduction} and Proposition~\ref{prop:BP3},
$g'_{1}$ and $g'_{2}$ are intersecting-supermodular functions on $U\setminus K$ 
satisfying $|X|\geq g'_{i}(X)$ for every $i\in \{1,2\}$ and $X\in \F'_{i}$.
Since $|U\setminus K|<|U|$, the inductive assumption implies that there exist
$\pi'_{1}, \pi'_{2}:U\setminus K\to \N$ such that (i), (ii), and (iii) hold with
$(U\setminus K, \pi'_{1}, \pi'_{2}, g'_{1}, g'_{2})$ in place of
$(U, \pi_{1}, \pi_{2}, g_{1}, g_{2})$. 
By (i) and (iii), for every $u\in U\setminus K$, we have
\begin{align}
& \pi'_{1}(u)+\pi'_{2}(u)-1\leq \max\{d[g'_{1}](u), ~d[g'_{2}](u)\},\label{eq:induction1}\\
& \pi'_{i}(u)\leq d[g'_{i}](u)~~~(i\in \{1,2\}).\label{eq:induction3}
\end{align}
By (ii), for each $i\in \{1,2\}$, we have
\vspace{-3mm}
\begin{align}
& |\pi'_{i}(X)|\geq g'_{i}(X)~~~~(X\in \F'_{i}).\label{eq:induction2}
\end{align}
By the definition of the reduction $g'_{i}$, for each  $i\in \{1,2\}$, we have
\begin{align}
 g'_{i}(Z)&\geq g_{i}(Z) \text{~~~~~~~~~} (Z\in \F_{i}, ~Z\cap K=\emptyset), \label{eq:gK01}\\
 g'_{i}(Z\setminus K)&\geq g_{i}(Z)-1\text{~~~~} (Z\in \F_{i}, ~Z\cap K\neq\emptyset).\label{eq:gK02}
\end{align}
Also, since $K$ is a common partial transversal, 
Proposition~\ref{prop:BP4} implies that,
for every $i\in \{1,2\}$ and $u\in U\setminus K$, we have  
\vspace{-3mm}
\begin{align}
d[g'_{i}](u)&\leq d[g_{i}](u),\label{eq:reduction1}\\
P[g_{i}](u)\cap K\neq\emptyset &\implies d[g'_{i}](u)<d[g_{i}](u).\label{eq:reduction2}
\end{align}
Recall that $K$ satisfies (a) or (b).
We now show the case in which (a) holds (the case for (b) is shown analogously).  
Then, for every $u\in U$, we have
\begin{equation}
P[g_{1}](u)\cap K\neq \emptyset \implies P[g_{2}](u)\cap K\neq \emptyset.\label{eq:bipartite}
\end{equation}
Using the functions $\pi'_{1}, \pi'_{2}:U\setminus K\to \N$, we define $\pi_{1}, \pi_{2}:U\to \N$ by
\begin{align*}
&\pi_{1}(u)=
\begin{cases}
1& (u\in K),\\
\pi'_{1}(u)&  (u\in U\setminus K,~P[g_{1}](u)\cap K=\emptyset),\\
\pi'_{1}(u)+1&    (u\in U\setminus K,~P[g_{1}](u)\cap K\neq\emptyset),
\end{cases}\\
~~~\\
&\pi_{2}(u)=
\begin{cases}
d[g_{2}](u)& (u\in K),\\
\pi'_{2}(u)&  (u\in U\setminus K).
\end{cases}
\end{align*}
We check that these $\pi_{1}$ and $\pi_{2}$ satisfy conditions (i), (ii), and (iii).

\paragraph{Conditions (i) and (iii):}
There are three cases corresponding to the definition of $\pi_{1}(u)$.
\smallskip

If $u\in K$, then we have $\pi_{1}(u)=1$ and $\pi_{2}(u)=d[g_{2}](u)$.
Then, $\pi_{1}(u)+\pi_{2}(u)-1=d[g_{2}](u)\leq \max\{d[g_{1}](u), ~d[g_{2}](u)\}$.
Also,  $\pi_{i}(u)\leq d[g_{i}](u)$ for each $i\in \{1,2\}$. 
\medskip

If $u\in U\setminus K$ and $P[g_{1}](u)\cap K=\emptyset$,
then $\pi_{1}(u)=\pi'_{1}(u)$ and $\pi_{2}(u)=\pi'_{2}(u)$.
By \eqref{eq:induction1} and \eqref{eq:reduction1}, we have
$\pi_{1}(u)+\pi_{2}(u)-1=\pi'_{1}(u)+\pi'_{2}(u)-1\leq \max\{d[g'_{1}](u), ~d[g'_{2}](u)\}\leq \max\{d[g_{1}](u), ~d[g_{2}](u)\}$.
Also, \eqref{eq:induction3} and \eqref{eq:reduction1} imply  
$\pi_{i}(u)\leq d[g_{i}](u)$ for each $i\in \{1,2\}$. 
\medskip

If $u\in U\setminus K$ and $P[g_{1}](u)\cap K\neq\emptyset$,
then $\pi_{1}(u)=\pi'_{1}(u)+1$ and $\pi_{2}(u)=\pi'_{2}(u)$.
From \eqref{eq:bipartite} and \eqref{eq:reduction2},
we have $d[g'_{1}](u)<d[g_{1}](u)$ and $d[g'_{2}](u)<d[g_{2}](u)$,
and hence
\[\max\{d[g'_{1}](u), ~d[g'_{2}](u)\}<\max\{d[g_{1}](u), ~d[g_{2}](u)\}.\]
With \eqref{eq:induction1}, this implies
$\pi_{1}(u)+\pi_{2}(u)-1=(\pi'_{1}(u)+\pi'_{2}(u)-1)+1\leq \max\{d[g'_{1}](u), ~d[g'_{2}](u)\}+1\leq \max\{d[g_{1}](u), ~d[g_{2}](u)\}$.
Also, \eqref{eq:induction3} and $d[g'_{i}](u)<d[g_{i}](u)$ imply  
$\pi_{i}(u)\leq d[g_{i}](u)$ for each $i$. 

\paragraph{Condition (ii):}
We  show $|\pi_{i}(Z)|\geq g_{i}(Z)$ for any $i\in \{1,2\}$ and $Z\in \F_{i}$.
If $g_{i}(Z)\leq 0$, the claim is clear.
Also, if $g_{i}(Z)=1$, then the assumption $|Z|\geq g_{i}(Z)$ implies $Z\neq \emptyset$ and hence
$|\pi_{i}(Z)|\geq 1=g_{i}(Z)$. 
Therefore, let us assume $g_{i}(Z)\geq 2$.

By Claims~\ref{claim:BP01} and \ref{claim:BP02}, then
there exist $\hat{Z}\in \E[g_{i}]$ and $\hat{P}\in \BP[g_{i}]\cap \E[g_{i}]$ 
satisfying 
\begin{equation}
\hat{Z}\subseteq Z,~~\hat{Z}\subseteq \hat{P},~~g_{i}(\hat{Z})\geq g_{i}(Z).\label{eq:color}
\end{equation}
(In particular, $\hat{Z}=Z$ if $Z\in \E[g_{i}]$.)
We need to consider cases $i=1$ and $i=2$ separately.

\medskip
\noindent\underline{Case $i=1$.}
Note that $\hat{P}$ in \eqref{eq:color} satisfies $\hat{P}=P[g_{1}](u)$ for all $u\in \hat{Z}\subseteq \hat{P}$.
Therefore, by the definition of $\pi_{1}$, if $\hat{P}\cap K=\emptyset$, then
$\pi_{1}(u)=\pi'_{1}(u)$ for all $u\in \hat{Z}\setminus K$.
Also, if $\hat{P}\cap K\neq \emptyset$, then
$\pi_{1}(u)=\pi'_{1}(u)+1$ for all $u\in \hat{Z}\setminus K$.
Thus, in both cases, we have 
\begin{equation}
|\pi_{1}(\hat{Z}\setminus K)|=|\pi'_{1}(\hat{Z}\setminus K)|.\label{eq:color1}
\end{equation}

If $\hat{Z}\cap K=\emptyset$, then \eqref{eq:color1} means 
$|\pi_{1}(\hat{Z})|=|\pi'_{1}(\hat{Z})|$. 
With \eqref{eq:induction2}, \eqref{eq:gK01}, \eqref{eq:color},
this implies the following inequality. (Note that the monotonicity of $|\pi_{1}(\cdot)|$ and $\hat{Z}\subseteq Z$ imply $|\pi_{1}(\hat{Z})|\leq |\pi_{1}(Z)|$.)
\[|\pi_{1}(Z)|\geq |\pi_{1}(\hat{Z})|=|\pi'_{1}(\hat{Z})|\geq g'_{1}(\hat{Z})\geq g_{1}(\hat{Z})\geq g_{1}(Z).\]

If $\hat{Z}\cap K\neq \emptyset$, then, as mentioned above, $\pi_{1}(u)=\pi'_{1}(u)+1>1$ for all $u\in \hat{Z}\setminus K$.
This implies $1\not\in \pi_{1}(\hat{Z}\setminus K)$.
Since $\pi_{1}(u)=1$ for any $u\in \hat{Z}\cap K\neq \emptyset$,
we have $|\pi_{1}(\hat{Z})|=|\pi_{1}(\hat{Z}\setminus K)|+1$. 
With \eqref{eq:induction2}, \eqref{eq:gK02}, \eqref{eq:color}, \eqref{eq:color1}, this implies
\[|\pi_{1}(Z)|\geq |\pi_{1}(\hat{Z})|=|\pi_{1}(\hat{Z}\setminus K)|+1=|\pi'_{1}(\hat{Z}\setminus K)|+1\geq g'_{1}(\hat{Z}\setminus K)+1\geq g_{1}(\hat{Z})\geq g_{1}(Z).\]

\noindent\underline{Case $i=2$.}
By the definition of $\pi_{2}$, we have $\pi_{2}(u)=\pi'_{2}(u)$ for every $u\in U\setminus K$.
Hence,
\begin{equation}
\pi_{2}(\hat{Z}\setminus K)=\pi'_{2}(\hat{Z}\setminus K).\label{eq:color3}
\end{equation}

If $\hat{Z}\cap K=\emptyset$, then \eqref{eq:color3} implies 
$|\pi_{2}(\hat{Z})|=|\pi'_{2}(\hat{Z})|$.
With \eqref{eq:induction2}, \eqref{eq:gK01}, \eqref{eq:color}, this implies 
\[|\pi_{2}(Z)|\geq |\pi_{2}(\hat{Z})|=|\pi'_{2}(\hat{Z})|\geq g'_{2}(\hat{Z})\geq g_{2}(\hat{Z})\geq g_{2}(Z).\]

If $\hat{Z}\cap K\neq \emptyset$, then $\hat{P}$ in \eqref{eq:color} satisfies $\hat{Z}\cap K\subseteq \hat{P}\cap K\neq \emptyset$.
For every $u\in \hat{Z}\setminus K\subseteq \hat{P}$,
since $P[g_{2}](u)=\hat{P}$ holds,
\eqref{eq:reduction2} implies $d[g'_{2}](u)<d[g_{2}](u)$.
Also, $d[g_{2}](u)=g_{2}(\hat{P})$ by Proposition~\ref{prop:BP2}.
By \eqref{eq:induction3}, then every $u\in \hat{Z}\setminus K$ satisfies
\[\pi_{2}(u)=\pi'_{2}(u)\leq d[g'_{2}](u)<d[g_{2}](u)=g_{2}(\hat{P}),\]
and hence 
$g_{2}(\hat{P})\not\in \pi_{2}(\hat{Z}\setminus K)$.
Since any $\hat{u}\in \hat{Z}\cap K\subseteq \hat{P}$
satisfies $\pi_{2}(\hat{u})=d[g_{2}](\hat{u})=g_{2}(\hat{P})$ by the definition of $\pi_{2}$, 
this implies $|\pi_{2}(\hat{Z})|=|\pi_{2}(\hat{Z}\setminus K)|+1$.
By \eqref{eq:induction2}, \eqref{eq:gK02}, \eqref{eq:color}, \eqref{eq:color3}, then
\[|\pi_{2}(Z)|\geq |\pi_{2}(\hat{Z})|=|\pi_{2}(\hat{Z}\setminus K)|+1\geq |\pi'_{2}(\hat{Z}\setminus K)|+1\geq g'_{2}(\hat{Z}\setminus K)+1\geq g_{2}(\hat{Z})\geq g_{2}(Z)\]
follows.
\end{proof}

\section{Properties of Bunch Partitions}\label{sec:BP}
This section shows Propositions~\ref{prop:BP}--\ref{prop:BP4},
which state properties of bunch partitions.
The first subsection gives some basic properties and proves Propositions~\ref{prop:BP} and \ref{prop:BP2}. 
The second shows properties related to reduction by partial transversals
and proves Propositions~\ref{prop:BP3} and \ref{prop:BP4}.

\subsection{Basic Properties}
Let $g:\F\to \Z$ be an intersecting-supermodular function on $U$. 
Recall that the effective set family $\E[g]$ 
is defined as the family of subsets $X\in \F$ satisfying 
\begin{align}
&~~g(X)\geq 2,\label{eq:2}\\
&\not\exists X'\in \F:[X'\subsetneq X, ~ g(X')\geq g(X)].\label{eq:effec}
\end{align}

\begin{claim}\label{claim:intersect}
If $X, Y\in\E[g]$ are intersecting,  $g(X\cup Y)>\max\{g(X), g(Y)\}$ and $X\cup Y\in \E[g]$. 
\footnote{In contrast,  $\E[g]$ is not closed under taking intersection.}
\end{claim}
\begin{proof}
As $X, Y\in \E[g]$ are intersecting, we have $g(X)+g(Y)\leq g(X\cup Y)+g(X\cap Y)$.
Also, as $X\cap Y\subsetneq X$, \eqref{eq:effec} for $X$ implies $g(X)>g(X\cap Y)$.
Thus, we obtain $g(Y)<g(X\cap Y)$. 
Similarly, from \eqref{eq:effec} for $Y$, we obtain $g(X)<g(X\cap Y)$.
Thus, $g(X\cup Y)>\max\{g(X), g(Y)\}\geq 2$.

Let  $Z$ be a minimal maximizer of $g$ in $\set{Z'| Z'\in \F,~Z'\subseteq X\cup Y}$. 
To show $X\cup Y\in \E[g]$, it suffices to prove $Z=X\cup Y$.
Suppose, to the contrary, we have $Z\subsetneq X\cup Y$,
which implies either $X\setminus Z\neq \emptyset$ or $Y\setminus Z\neq \emptyset$.
Without loss of generality, let $X\setminus Z\neq \emptyset$.
By definition, $Z$ satisfies $g(Z)\geq g(X\cup Y)>\max\{g(X), g(Y)\}$. 
This implies $Z\not\subseteq X$ and $Z\not\subseteq Y$
since $X$ and $Y$ satisfy \eqref{eq:effec}.
Combining $Z\subseteq X\cup Y$ and $Z\not\subseteq Y$ implies $X\cap Z\neq \emptyset$.
Thus, we have
\[X\setminus Z\neq \emptyset,~~Z\not\subseteq X,~~X\cap Z\neq \emptyset,\]
which mean that $X$ and $Z$ are intersecting.
Then, $g(X)+g(Z)\leq g(X\cup Z)+g(X\cap Z)$.
As $g(X)>g(X\cap Z)$ by \eqref{eq:effec} for $X$, we obtain $g(Z)<g(X\cup Z)$, 
which contradicts the fact that $Z$ is a maximizer.
\end{proof}

Recall that $\BP[g]$ consists of all maximal members of $\E[g]$ and 
the singleton sets $\{u\}$ of all $u\in U\setminus \bigcup \E[g]$.
Then, Claim~\ref{claim:intersect} implies Proposition~\ref{prop:BP} as follows.
\begin{proof}[\bf\em Proof of Proposition~\ref{prop:BP}]
We show that $\BP[g]$ is a partition. 
By the definition of $\BP[g]$, it suffices to show that
the maximal members of $\E[g]$ are all pairwise disjoint.
Suppose, to the contrary, that
distinct  $X, Y\in \E[g]$ are maximal and not disjoint.
Then, they are intersecting, and Claim~\ref{claim:intersect}
implies $X\cup Y\in \E[g]$, which contradicts the maximality of $X$ and $Y$. 
\end{proof}

Recall that $d[g]:U\to \N$ is defined by  
$d[g](u)=\max\{1, \max\set{g(X)|u\in X\in \E[g]}\}$.
Also, recall that $P[g](u)$ is defined as the unique part of $\BP[g]$ 
containing $u$.

\begin{proof}[\bf\em Proof of Proposition~\ref{prop:BP2}]
If $u\in \bigcup \E[g]$, then $P[g](u)$ is a maximal member of $\E[g]$.
Since all maximal members of $\E[g]$ are pairwise disjoint by Proposition~\ref{prop:BP},
any $X\in \E[g]$ with $X\cap P[g](u)\neq \emptyset$ satisfies $X\subseteq P[g](u)$.
Because $u\in P[g](u)$, then the condition $u\in X\in \E[g]$ implies $X\subseteq P[g](u)$,
from which $g(X)\leq g(P[g](u))$ follows because $P[g](u)\in \E[g]$.
Thus, $d[g](u)=\max\set{g(X)|u\in X\in \E[g]}=g(P[g](u))$.
Also $g(P[g](u))\geq 2$ by $P[g](u)\in \E[g]$.

If $u\not\in \bigcup \E[g]$, 
then the claim immediately follows from the definitions of $\BP[g]$ and $d[g]$.
\end{proof}

\subsection{Reduction by a Partial Transversal}
As before, let $g:\F\to \Z$ be an intersecting-supermodular function. 
We also assume 
\begin{equation}
|Z|\geq g(Z)\quad(Z\in \F).\label{eq:assum}
\end{equation}
Under this assumption, \eqref{eq:2} implies the following observation.
\begin{claim}\label{claim:2}
Every $Z\in \E[g]$ satisfies $|Z|\geq 2$.
\finbox
\end{claim}

Take a partial transversal $K$ of $\BP[g]$
and let $g_{K}:\F_{K}\to \Z$ be the reduction of $g$ by $K$, i.e.,
\begin{equation}
g_{K}(X)=\max\set{\hat{g}_{K}(Z)|Z\in \F,~Z\setminus K=X}, \label{eq:gK}
\end{equation}
where $\hat{g}_{K}(Z)=g(Z)-1$ for $Z\in \F$ with $Z\cap K\neq\emptyset$
and $\hat{g}_{K}(Z)=g(Z)$ for $Z\in \F$ with $Z\cap K=\emptyset$.
We often use the following observation. 
\begin{claim}
For any $Z', Z\in \F$, we have
\begin{align}
 [~Z'\subseteq Z,~ g(Z')\geq g(Z)~] &\implies \hat{g}_{K}(Z')\geq \hat{g}_{K}(Z),\label{eq:ghat1}\\
  g(Z')>g(Z) &\implies \hat{g}_{K}(Z')\geq \hat{g}_{K}(Z).  \label{eq:ghat2}
\end{align}
\end{claim}
\begin{proof}
For \eqref{eq:ghat1}, observe that $Z'\subseteq Z$ implies $[Z'\cap K\neq \emptyset \implies Z\cap K\neq \emptyset]$.
\eqref{eq:ghat2} is clear.
\end{proof}

\begin{proof}[\bf\em Proof of Proposition~\ref{prop:BP3}]
We show $|X|\geq g_{K}(X)$ for any $X\in \F_{K}$. 
By the definition of $g_{K}$, it suffices to show
$|Z\setminus K|\geq \hat{g}_{K}(Z)$ for every $Z\in \F$.
If $Z\cap K=\emptyset$, then \eqref{eq:assum} immediately implies
$|Z\setminus K|=|Z|\geq g(Z)=\hat{g}_{K}(Z)$.
Also, if $Z\cap K\neq \emptyset$ and $g(Z)\leq 1$, then  $\hat{g}_{K}(Z)\leq 0$,
and hence clearly $|Z\setminus K|\geq \hat{g}_{K}(Z)$.
Therefore, we assume $Z\cap K\neq \emptyset$ and $g(Z)\geq 2$.

In the case $Z\in \E[g]$,
by Claim~\ref{claim:BP01}, there is $P\in \BP[g]\cap \E[g]$ with $Z\subseteq P$.
Since $K$ is a partial transversal of $\BP[g]$, we have $|Z\cap K|\leq |P\cap K|\leq 1$.
By \eqref{eq:assum} and $Z\cap K\neq \emptyset$, 
then $|Z\setminus K|=|Z|-1\geq g(Z)-1=\hat{g}_{K}(Z)$.

In the case $Z\not\in \E[g]$,
Claim~\ref{claim:BP02} implies 
that some $Z'\in \F$ and $P\in \BP[g]\cap \E[g]$ satisfy $Z'\subseteq Z\cap P$ and $g(Z')\geq g(Z)$.
As $K$ is a partial transversal of $\BP[g]$, 
we have $|Z'\cap K|\leq |P\cap K|\leq 1$. 
With \eqref{eq:assum} for $Z'$, this implies 
$|Z\setminus K|\geq |Z'\setminus K|\geq |Z'|-1\geq g(Z')-1\geq g(Z)-1=\hat{g}_{K}(Z)$.
\end{proof}

Let us consider the effective set family $\E[g_{K}]\subseteq \F_{K}$ 
and the bunch partition $\BP[g_{K}]$ 
defined for the reduction $g_{K}$.
They are families on $U\setminus K$. 
Note that they do not necessarily coincide with
$\set{Z\setminus K|Z\in \E[g]}$ and
$\set{P\setminus K|P\in \BP[g]}$
(see Example~\ref{ex:example1}).

To show Proposition~\ref{prop:BP4}, we prepare the following five claims.

\begin{claim}\label{claim:refine0}
For every $X\in \E[g_{K}]$, there exists $P\in \BP[g]\cap \E[g]$ such that $X\subseteq P\setminus K$.
\end{claim}
\begin{proof}
By \eqref{eq:gK}, there is $Z\in \F$ satisfying $g_{K}(X)=\hat{g}_{K}(Z)$ and $Z\setminus K=X$.
Suppose to the contrary,  that $X\not\subseteq P\setminus K$ for every $P\in \BP[g]\cap \E[g]$.
Since $Z\setminus K=X$, this implies $Z\not\subseteq P$ for every $P\in \BP[g]\cap \E[g]$.
Then, we have $Z\not\in \E[g]$ by the definition of $\BP[g]$.
Since $X\in \E[g_{K}]$ implies $2\leq g_{K}(X)=\hat{g}_{K}(Z)\leq g(Z)$,
Claim~\ref{claim:BP02} implies that there exist $Z'\in \E[g]$ and $P'\in \BP[g]\cap \E[g]$ with 
$Z'\subseteq Z\cap P'$ and $g(Z')\geq g(Z)$.
Then, we have
\[Z'\setminus K\subseteq Z\setminus K=X,~~~~Z'\setminus K\subseteq P'\setminus K,~~~~X\not\subseteq P'\setminus K.\] 
This implies  $Z'\setminus K \subsetneq X$,
and hence $g_{K}(Z'\setminus K)<g_{K}(X)$ follows from  $X\in \E[g_{K}]$.
However, \eqref{eq:gK}, $g(Z')\geq g(Z)$ and \eqref{eq:ghat1} imply
$g_{K}(Z'\setminus K)\geq \hat{g}_{K}(Z')\geq \hat{g}_{K}(Z)=g_{K}(X)$, a contradiction.
\end{proof}

By Claim~\ref{claim:refine0}, we can observe the following structural property of $\BP[g_{K}]$.
\begin{claim}\label{claim:refine1}
$\BP[g_{K}]$ is a refinement of the partition $\set{P\setminus K|P\in \BP[g]}$ of $U\setminus K$.
\finbox
\end{claim}

\begin{claim}\label{claim:ZB}
For $Z\in \F$ and $P\in \BP[g]\cap\E[g]$ with
$Z\setminus K\subseteq P\setminus K$,
we have $\hat{g}_{K}(Z)\leq\hat{g}_{K}(P)$.\\
In particular, if $Z\setminus K\subsetneq P\setminus K$ and $P\cap K=\emptyset$, 
then $\hat{g}_{K}(Z)<\hat{g}_{K}(P)$.
\end{claim}
\begin{proof}
Take $Z\in \F$ and $P\in \BP[g]\cap\E[g]$ with $Z\setminus K\subseteq P\setminus K$.
Note that $P\in \E[g]$ implies $g(P)\geq 2$ by \eqref{eq:2}, and hence $\hat{g}_{K}(P)\geq 1$.
Then, it suffices to consider the case $\hat{g}_{K}(Z)\geq 2$.
Therefore, we assume $g(Z)\geq 2$.

As $Z\in \F$,
Claims~\ref{claim:BP01} and \ref{claim:BP02} imply that there are $Z'\in \E[g]$ and $P'\in \BP[g]\cap \E[g]$ 
satisfying $Z'\subseteq Z\cap P'$ and $g(Z')\geq g(Z)$.
(In particular, $Z'=Z$ if $Z\in \E[g]$.)
As $K$ is a partial transversal,  $|Z'\cap K|\leq |P'\cap K|\leq 1$.
Also, $Z'\in \E[g]$ implies $|Z'|\geq 2$ by Claim~\ref{claim:2}.
Thus, we have
\[\emptyset\neq Z'\setminus K\subseteq Z\setminus K\subseteq P\setminus K,~~~~
Z'\subseteq P'.\]
As $\BP[g]$ is a partition, these two imply $P=P'$.
Thus, $Z'\subseteq P$, and hence \eqref{eq:effec} for $P\in \E[g]$ implies either $Z'=P$ or $g(Z')<g(P)$.
In the case $Z'=P$, we have $P=Z'\subseteq Z$ and $g(P)=g(Z')\geq g(Z)$, which imply 
$\hat{g}_{K}(Z)\leq\hat{g}_{K}(P)$ by \eqref{eq:ghat1}.
In the case $g(Z')<g(P)$, we have $\hat{g}_{K}(Z)\leq\hat{g}_{K}(P)$ by \eqref{eq:ghat2}.
Thus, the first statement is shown.

For the second statement, assume
$Z\setminus K\subsetneq P\setminus K$.
By $Z'\subseteq Z$, this implies $Z'\setminus K\subsetneq P\setminus K$, and hence $Z'\neq P$. 
By the above argument, then $g(Z')<g(P)$. 
Therefore, when $P\cap K=\emptyset$, we have
$\hat{g}_{K}(Z)\leq g(Z)\leq g(Z')<g(P)=\hat{g}_{K}(P)$.
Thus, $\hat{g}_{K}(Z)<\hat{g}_{K}(P)$ is obtained.
\end{proof}

\begin{claim}\label{claim:ZB2}
For every $P\in \BP[g]\cap\E[g]$ with $P\cap K=\emptyset$, we have $P\in \BP[g_{K}]\cap\E[g_{K}]$
and $g_{K}(P)=g(P)$.
\end{claim}
\begin{proof}
Since $P\cap K=\emptyset$, we have $P\setminus K=P$ and $\hat{g}_{K}(P)=g(P)$.
By \eqref{eq:gK}, the value $g_{K}(P)$ is defined as 
$g_{K}(P)=\max\set{\hat{g}_{K}(Z)|Z\in \F,~Z\setminus K=P}$.
Here, the maximum of the right-hand side is attained by $P$ itself because
any $Z\in \F$ with $Z\setminus K=P=P\setminus K$
satisfies $\hat{g}_{K}(Z)\leq\hat{g}_{K}(P)$ by the first statement of Claim~\ref{claim:ZB}.
Thus, $g_{K}(P)=\hat{g}_{K}(P)=g(P)$.
Also, the second statement of Claim~\ref{claim:ZB} implies that,
for any $X\in \F_{K}$ with $X\subsetneq P\setminus K=P$, we have 
$g_{K}(X)=\max\set{\hat{g}_{K}(Z)|Z\in \F,~Z\setminus K=X}<\hat{g}_{K}(P)=g_{K}(P)$.
Hence, $P$ belongs to $\E[g_{K}]$. 
Also, by Claim~\ref{claim:refine1} and $P\setminus K=P\in \BP[g]$,
$P$ is maximal in $\E[g_{K}]$. Thus, $P\in \BP[g_{K}]\cap \E[g_{K}]$. 
\end{proof}

\begin{claim}\label{claim:ZB3}
For every $P\in \BP[g]\cap\E[g]$ with $P\cap K\neq\emptyset$ and
$X\in \F_{K}$ with $X\subseteq P\setminus K$, we have $g_{K}(X)<g(P)$.
\end{claim}
\begin{proof}
Since $P\cap K\neq\emptyset$, we have $\hat{g}_{K}(P)=g(P)-1<g(P)$.
By \eqref{eq:gK}, we have
$g_{K}(X)=\max\set{\hat{g}_{K}(Z)|Z\in \F,~Z\setminus K=X}$,
which is at most $\hat{g}_{K}(P)$ because
$Z\setminus K=X\subseteq P\setminus K$ implies $\hat{g}_{K}(Z)\leq\hat{g}_{K}(P)$ by the first statement of Claim~\ref{claim:ZB}.
Therefore, $g_{K}(X)\leq \hat{g}_{K}(P)<g(P)$.
\end{proof}

We are now ready to show Proposition~\ref{prop:BP4}.
\begin{proof}[\bf\em Proof of Proposition~\ref{prop:BP4}]
Take any $u\in U\setminus K$.
If $u\in U\setminus \bigcup \E[g]$, then $u\in U\setminus \bigcup \E[g_{K}]$ 
by Claim~\ref{claim:refine0}.
Then, $P[g_{K}](u)=P[g](u)=\{u\}$, and hence $P[g_{K}](u)\cap K=\{u\}\cap K=\emptyset$ because $u\in U\setminus K$.
We also have $d[g_{K}](u)=d[g](u)=1$. 
Thus, the claim is proved for this case.
Therefore, we assume $P[g](u)\in \E[g]$. 
\smallskip

In the case $P[g](u)\cap K=\emptyset$,  
Claim~\ref{claim:ZB2} implies $P[g_{K}](u)=P[g](u)$ and $g_{K}(P[g](u))=g(P[g](u))$.
By Proposition~\ref{prop:BP2}, then we have 
$d[g_{K}](u)=g_{K}(P[g_{K}](u))=g(P[g](u))=d[g](u)$.
Thus, $d[g_{K}](u)=d[g](u)$ is obtained.
\smallskip

In the case $P[g](u)\cap K\neq \emptyset$, we have either
$P[g_{K}](u)=\{u\}$ or $P[g_{K}](u)\in \E[g_{K}]$. 
If $P[g_{K}](u)=\{u\}$, then $d[g_{K}](u)=1<2\leq g(P[g](u))=d[g](u)$
follows from Proposition~\ref{prop:BP2} and $P[g](u)\in \E[g]$.
Thus, $d[g_{K}](u)<d[g](u)$.
If $P[g_{K}](u)\in \E[g_{K}]$, then
$u\in P[g](u)\cap P[g_{K}](u)$ and Claim~\ref{claim:refine1}
imply $P[g_{K}](u)\subseteq P[g](u)\setminus K$.
Then,  $g_{K}(P[g_{K}](u))<g(P[g](u))$ by Claim~\ref{claim:ZB3}.
By Proposition~\ref{prop:BP2}, this means $d[g_{K}](u)<d[g](u)$.
\end{proof}

\begin{remark}
The results of this paper can be extended to the setting of skew-supermodular coloring \cite{FK09}.
To obtain the key lemma (Lemma~\ref{lem:key}), 
we used intersecting supermodularity directly only 
in the proof of Claim~\ref{claim:intersect}.
We can observe that this claim is also true for skew-supermodular functions.
Moreover, it is shown in Iwata and Yokoi \cite{IY17}
that Proposition~\ref{prop:IY} and Claim~\ref{claim:reduction} 
are true for skew-supermodular functions.
Therefore, Theorem~\ref{thm:main} can extends to skew-supermodular functions.
\end{remark}

\section{Acknowledgments}
I would like to thank Andr\'{a}s Frank for his valuable comments and questions,
which motivated me to work on this subject.
I gratefully acknowledge Tam\'{a}s Kir\'{a}ly, Krist\'{o}f B\'{e}rczi, and Satoru Iwata for their useful comments.
This work was supported by JST CREST, Grant Number JPMJCR14D2, Japan.

\end{document}